\documentclass[11pt]{article}
\usepackage{
    amsmath,  amsfonts, amssymb,  amsthm,   amscd, 
    gensymb,  graphicx, comment,  etoolbox, url,
    booktabs, stackrel, mathtools,enumitem, mathdots,  microtype, lmodern,    mathrsfs, graphicx, tikz,  longtable,tabularx, float, tikz, pst-node, tikz-cd, multirow, tabularx, amscd,  bm, array, makecell, diagbox, booktabs,ragged2e, caption, subcaption }

    \usepackage{makecell,slashbox}
\usepackage{xcolor}
\usepackage[all]{xy}
\usepackage[pagebackref=true, colorlinks=true, linkcolor=blue, citecolor=blue, urlcolor=blue, breaklinks=true]{hyperref}
\usepackage[utf8]{inputenc}
\usepackage{microtype, fullpage, wrapfig,textcomp,mathrsfs,csquotes,fbb}
\usepackage[colorlinks=true, linkcolor=blue, citecolor=blue, urlcolor=blue, breaklinks=true]{hyperref}

\usepackage[capitalise]{cleveref}
\usepackage{todonotes}
\usetikzlibrary{positioning}
\usetikzlibrary{shapes,arrows.meta,calc}
\usetikzlibrary{arrows}

\theoremstyle{plain}
\newtheorem{theorem}{Theorem}[section]

\newtheorem{prop}[theorem]{Proposition}
\newtheorem{lemma}[theorem]{Lemma}

\theoremstyle{definition}
\newtheorem{definition}[theorem]{Definition}
\newtheorem{remark}[theorem]{Remark}

\newtheorem{example}[theorem]{Example}

\newcommand{\CC} {\mathbb C}
\newcommand{\ZZ} {\mathbb Z}

\newcommand{\QQ} {\mathbb Q}

\def\id{\mathrm{id}}
\def\H{\mathrm{H}}

\def\c#1{\mathcal{#1}}

\def\u#1{\underline{#1}}

\def\lim{\mathrm{lim}}

\def\xra{\xrightarrow}
\def\xla{\xleftarrow}
\newcommand{\TC}{\mathrm{TC}}
\newcommand{\cat}{\mathrm{cat}}

\def\secat#1#2{\mathrm{secat}_{#1}(#2)}

\def\sc#1#2{\mathrm{sc}_{#1}(#2)}
\def\msc#1#2{\mathrm{msc}_{#1}(#2)}
\def\hsc#1#2{\mathrm{Hsc}_{#1}(#2)}

\def\img{\mathrm{img}}
\def\nil{\mathrm{nil}}

\begin{document}

	\title{Rational relative sectional category}

	\author{ Lekha Das\footnote{Department of Mathematics, Indian Institute of Technology Bombay, India. 22d0790@iitb.ac.in}, Bittu Singh\footnote{Department of Mathematics, Indian Institute of Technology Bombay, India. 22d0781@iitb.ac.in} }
	\maketitle

	\begin{abstract}
    We develop an algebraic model for the relative sectional category of a continuous map in rational homotopy theory using commutative differential graded algebras (CDGAs). 
    Our main result establishes that for formal maps, the rational relative sectional category can be computed purely from cohomology, using ideal nilpotency.
    We also show that this equality may fail in general topological settings. Applying this framework, we obtain purely algebraic characterizations for the rational Lusternik-Schnirelmann category and the rational higher topological complexity of a map. 
   Finally, we provide an algebraic description of the rational homotopic distance between formal maps.
	\end{abstract}


	
	\vspace{.1in}
\section{Introduction}
The notion of \emph{sectional category} was introduced by Schwarz in his seminal paper \cite{Schwarz1966} as a numerical homotopy invariant associated to a fibration or, more generally, to a continuous map. Informally, the sectional category of a map $f\colon E\to B$ measures the minimal number of open sets required to cover $B$, such that $f$ admits a local homotopy section on each set. Schwarz established several bounds for this invariant. 
A particularly important example is the \emph{Lusternik-Schnirelmann category} of a space $X$ \cite{LScat}, which can be recovered as the sectional category of the inclusion of a point into $X$.
Explicitly, the LS category of a space $X$, denoted by $\mathrm{cat}(X)$, is the smallest  non-negative integer $r\leq \infty$ such that $X$ can be covered by $r+1$ open subsets $V_0, V_1, \dots, V_r$, where each inclusion $V_i\xhookrightarrow{} X$ is null-homotopic.

Motivated by problems in robotics and motion planning, Farber introduced the invariant of topological complexity in a series of influential papers \cite{Farber2003, Farber2004, Farber2006, Farber2008}. For a path-connected space $X$, the \emph{topological complexity} $\TC(X)$ is defined as the sectional category of the free path fibration $X^I\to X\times X$ which assigns to each path $\gamma$ its pair of endpoints $(\gamma(0),\gamma(1))$. Subsequently, Rudyak and others \cite{ Rudyak2010, Rudyak_Gon_Higher} introduced a higher-order generalization of topological complexity for sequential motion planning. The \emph{$n$-th topological complexity} $\TC_n(X)$ is the sectional category of the fibration $\pi_n\colon X^I\to X^n$ given by $$\pi_n(\gamma)=\bigg(\gamma(0),\gamma\bigg(\frac{1}{n-1}\bigg), \dots, \gamma\bigg(\frac{n-2}{n-1}\bigg),\gamma(1)\bigg).$$

A powerful approach to understanding such homotopy invariants is provided by localization, and in particular, by rational homotopy theory, where spaces are studied up to rational equivalence; the details are given in \cite{FelixHalperinThomas2001}. 
In this rational setting, Félix and Halperin \cite{FelixHalperin1982} obtained a complete algebraic characterization of the Lusternik-Schnirelmann category of a rational space $X_0$ in terms of its Sullivan minimal model. Since the LS category and topological complexity are closely related, Farber posed the problem of finding an analogous algebraic description of the topological complexity of rational spaces; see \cite[Problem 31]{Farber2006}.

This problem has been investigated by several authors over the past two decades. A general framework for describing the sectional category of the rationalization of a fibration was developed by Fass\`o Velenik \cite{FassoVelenik2002}. Building on this framework, Lechuga and Murillo \cite{lechuga2007topological} gave the first characterization of rational topological complexity and showed that, for formal spaces, this invariant attains its cohomological lower bound. In \cite{FernandezSuarez2006}, the authors introduced $\mathrm{MTC}(X)$, a rational homotopy invariant providing a lower bound for $\TC(X_0)$. Further algebraic descriptions were later obtained by Jessup, Murillo, and Parent \cite{JessupMurilloParent2012}.
Carrasquel et al. developed these ideas further in a series of papers \cite{Carrasquel2015, Carrasquel2016a, Carrasquel2016b, Carrasquel2017a, Carrasquel2018}, including an upper bound for the rational sectional category.

While the rational topological complexity of spaces is well understood, many problems in motion planning involve constrained environments or maps between configuration spaces. The purpose of this paper is to extend these algebraic characterizations to the relative setting. We develop an algebraic description of the rational relative sectional category of a map in terms of its algebraic models.
\begin{theorem}[Theorem~\ref{thm: secat top vs algebra}]
    Let $f\colon K\to B$ and $p\colon E\to B$ be two maps, and let $f_0$ and $p_0$ denote their rationalizations respectively.
    Let $\psi\colon\c B\to \c E$ be a surjective model for $p_0$ and $\varphi\colon \c B\to \c K$ be any model for $f_0$.
     Then $$\mathrm{secat}_{f_0}(p_0)=\mathrm{secat}_{\varphi}(\psi).$$
\end{theorem}

Building on this structural foundation, we show that for formal maps,
the rational relative sectional category can be computed purely from cohomology using ideal nilpotency.
\begin{theorem}[Theorem~\ref{thm: TOPsecat for formal cdga map}]
    Let $f\colon K\to B$ and $p\colon E\to B$ be formal maps such that  $p^\ast$ is surjective. Then $$\mathrm{secat}_{f_0}(p_0)= \nil({f^\ast}(\ker{p^\ast})).$$
\end{theorem}
In particular, we show that the rational topological complexity of a formal map can be computed purely algebraically via cohomological data.
\begin{prop}[Proposition~\ref{prop: Tc for formal map}]
For a formal map $f\colon X \to Y$, the $n$-th topological complexity of the map $f_0$ is given by $$\TC_n(f_0)=\nil(({f^\ast})^{\otimes n}(\ker \mu_n)),$$ where $\mu_n\colon H^\ast (Y)^{\otimes n} \to H^\ast(Y)$ is the multiplication map.
\end{prop}
We further apply these ideas to other invariants, such as LS-category and homotopic distance. Using this theory, we carry out some explicit computations.

\section{Background}
\subsection{Rational homotopy theory}
In this paper, we use methods from rational homotopy theory to study
the relative sectional category and the higher topological complexity of maps in the rational setting.
Let  CDGA denote a commutative differential graded algebra over $\mathbb{Q}$. Unless stated otherwise, all CDGAs are assumed to be cohomologically $1$-connected and of finite type, and all spaces are simply connected CW complexes of finite type. When coefficients are not specified, cohomology is taken with rational coefficients. A CDGA $A$ is called cohomologically $1$-connected if $H^0(A)=\QQ$ and $H^1(A)=0.$ It is said to be of finite type if each $H^k(A)$ is a finite-dimensional $\mathbb{Q}$-vector space for  $k \geq 0$. A space $X$ is simply connected if it is path connected and $\pi_1(X)=0$. It is said to be of finite type if $\pi_k(X)\otimes \QQ$ is a finite-dimensional $\QQ$-vector space for all $k\geq 0.$

A space $X$ is called \emph{rational} if its homotopy groups $\pi_\ast(X)$ are rational vector spaces. For any simply connected CW complex $X$ of finite type, there exists a map $\rho_X \colon X \to X_0$, where $X_0$ is a rational space and the induced map $\pi_\ast(\rho_X)\otimes \mathbb{Q}$ is an isomorphism. The space $X_0$ is called the \emph{rationalization} of $X$ (see, for example, \cite{FelixHalperinThomas2001}).
Given a map $f \colon X \to Y$, there is an induced map $f_0 \colon X_0 \to Y_0$, called the \emph{rationalization of $f$}, satisfying $\rho_Y \circ f \simeq f_0 \circ \rho_X$. Throughout this paper, the notation $f_0$ is reserved for the rationalization of $f$.

A CDGA morphism $f \colon (A,d) \to (B,d)$ is a \emph{quasi-isomorphism} if it induces an isomorphism on cohomology.
A \emph{(Sullivan) minimal algebra} $(\Lambda V, d)$ is a free CDGA on a graded vector space $V$ such that $d(V)\subseteq\Lambda^{\geq 2}V$. Every simply connected space $X$ admits a minimal model, that is, a quasi-isomorphism $\rho$ from a minimal algebra to the polynomial de Rham algebra of $X$ $$\rho\colon(\Lambda V, d)\to A_{PL}(X).$$ This establishes a Quillen equivalence between the category of simply connected CW-complexes and the category of CDGAs. 

A \emph{(minimal) relative Sullivan algebra} is a CDGA of the form $(A\otimes \Lambda V,d)$ where $$d(V)\subseteq (A^+\otimes \Lambda V)+(A\otimes \Lambda ^{\geq 2} V).$$
Any CDGA map $\varphi\colon A \to B$ admits a factorization through a (minimal) relative Sullivan algebra as
\[\begin{tikzcd}
	{(A,d)} & {(B,d)} \\
	& {(A\otimes\Lambda V,d)}
	\arrow["\varphi", from=1-1, to=1-2]
	\arrow["i"', hook, from=1-1, to=2-2]
	\arrow["\rho"', from=2-2, to=1-2]
\end{tikzcd}\]
where $i$ is the inclusion, $\rho$ is a quasi-isomorphism and $\rho i=\varphi.$
The map $i$ is called the \emph{minimal Sullivan model} of $\varphi.$

Two CDGAs $A$ and $B$ are said to be \emph{weakly equivalent} if they are connected by a zig-zag of quasi-isomorphisms. Equivalently, $A$ and $B$ are weakly equivalent if there exist quasi-isomorphisms $$A \xleftarrow{\sim} \Lambda V \xrightarrow{\sim} B,$$ where $\Lambda V$ is a minimal algebra. Two CDGA maps $\varphi_1 \colon A_1 \to B_1$ and $\varphi_2 \colon A_2 \to B_2$ are \emph{weakly equivalent} if there exists a homotopy commutative diagram
\[
\begin{tikzcd}
	{A_1} & {\Lambda V} & {A_2} \\
	{B_1} & {\Lambda W} & {B_2}
	\arrow["{\varphi_1}"', from=1-1, to=2-1]
	\arrow["\sim"', from=1-2, to=1-1]
	\arrow["\sim", from=1-2, to=1-3]
	\arrow[from=1-2, to=2-2]
	\arrow["{\varphi_2}", from=1-3, to=2-3]
	\arrow["\sim"', from=2-2, to=2-1]
	\arrow["\sim", from=2-2, to=2-3]
\end{tikzcd}
\]
with $\Lambda V$ and $\Lambda W$ minimal algebras.

A CDGA $A$ is called \emph{formal} if $A$ is weakly equivalent to its cohomology ring $H^\ast(A).$ Similarly, a CDGA morphism $f\colon A\to B$ is called \emph{formal} if $f$ is weakly equivalent to $f^\ast\colon H^\ast(A)\to H^\ast(B).$

A CDGA $(A,d)$ is a \emph{model} for a space $X$ if it is weakly equivalent to $A_{PL}(X)$, and a CDGA map $\varphi \colon A \to B$ is a \emph{model} for a map $f \colon X \to Y$ if it is weakly equivalent to the induced map $A_{PL}(f) \colon A_{PL}(Y) \to A_{PL}(X)$. A space $X$ (respectively, a map $f$) is called \emph{formal} if $H^\ast(X)$ (respectively, $f^\ast$) is a model for $X$ (respectively, $f$).

In particular, for formal spaces, the rational homotopy type of a space is completely determined by its cohomology algebra, and similarly, a formal map is determined by its induced map on cohomology. This perspective motivates the study of whether invariants such as topological complexity admit explicit cohomological descriptions.

We now record some useful technical results that allow us to replace homotopy commutative diagrams by strictly commutative ones after suitable modifications of the maps. 
\begin{lemma}\label{lem: make homotopy to equality}
Let $i \colon A \to B$ be a cofibration and let $f \colon B \to C$, $g \colon A \to C$ be morphisms in a model category with $C$ a fibrant object such that $$f \circ i \simeq g
$$ via a right homotopy. Then there exists a map $f' \colon B \to C$ such that $f' \simeq f$ and $$f' \circ i = g.$$
\end{lemma}

\begin{proof}
Let $H \colon A \to C^I$ be a right homotopy such that $p_0 \circ H = g$ and $p_1 \circ H = f \circ i$, where $C^I$ denotes a path object for $C$. By replacing $H$ by a good homotopy $K$ if necessary, we may assume that the projections $p_0,p_1 \colon C^I \to C$ are acyclic fibrations as $C$ is fibrant.

Consider the commutative diagram
\[
\begin{tikzcd}
	A & {C^I} \\
	B & C
	\arrow["H", from=1-1, to=1-2]
	\arrow["i"', hook, from=1-1, to=2-1]
	\arrow["{p_1}", two heads, from=1-2, to=2-2]
	\arrow["K", dashed, from=2-1, to=1-2]
	\arrow["f"', from=2-1, to=2-2]
\end{tikzcd}
\]
By the lifting property, there exists a map $K \colon B \to C^I$ making the diagram commute. Define
$f' = p_0 \circ K.$
Then $f' \simeq p_1 \circ K = f$, and moreover $$f' \circ i = p_0 \circ K \circ i = p_0 \circ H = g,$$ as required.
\end{proof}
\begin{prop}\label{prop:lifting prop for relative sullivan}
Let the following diagram be homotopy commutative in the category of CDGAs over $\mathbb{Q}$, where $i$ is the inclusion of a relative Sullivan algebra and $p$ is a quasi-isomorphism:
\[
\begin{tikzcd}
	A & B \\
	{A \otimes \Lambda W} & C
	\arrow["\alpha", from=1-1, to=1-2]
	\arrow["i"', hook, from=1-1, to=2-1]
	\arrow["p", from=1-2, to=2-2]
	\arrow["\beta"', from=2-1, to=2-2]
\end{tikzcd}
\]
Then there exists a CDGA morphism $$\gamma \colon A \otimes \Lambda W \longrightarrow B$$ such that $$p \circ \gamma \simeq \beta
\qquad\text{and}\qquad
\gamma \circ i = \alpha.$$
\end{prop}

\begin{proof}
By Lemma~\ref{lem: make homotopy to equality}, there exists a map $\beta' \simeq \beta$ such that
$\beta' \circ i = p \circ \alpha.$
Applying the lifting lemma for relative Sullivan algebras \cite[Proposition~14.6]{FelixHalperinThomas2001}, we obtain a morphism $$\gamma \colon A \otimes \Lambda W \longrightarrow B$$ satisfying $\gamma \circ i = \alpha$ and $$p \circ \gamma \simeq \beta' \simeq \beta.$$
\end{proof}

\subsection{Relative sectional category}
The notion of \emph{relative sectional category} was introduced by Gonz\'alez, Grant, and Vandembroucq \cite{gonzalez2019hopf}. Its properties and its relationship with various notions of topological complexity of a map were further studied by Calcines \cite{Calc24}. We recall the definition.
\begin{definition}
Let $K\xra{f} B\xla p E $ be two maps. The sectional category of $p$ with respect to $f$ is the least integer $m \leq \infty$ such that $K$ admits an open cover $\{U_i\}_{i=0}^m$ with the property that, for each $i$, there exists a map $s_i \colon U_i \to E$ satisfying $$p \circ s_i \simeq f|_{U_i}.$$
\end{definition}
Note that replacing $f$ and $p$ by their respective rationalizations $f_0$ and $p_0$ yields the inequality
$$\secat{f_0}{p_0} \leq \secat{f}{p}.$$
The sectional category of a map $p \colon E \to B$ is recovered as a special case of the relative sectional category by taking $f = \id_B$, that is, $$\secat{}{p} = \secat{\id_B}{p}.$$

Pave\v{s}i\v{c} \cite{Pavesic, pavevsic2024topological} originally introduced the topological complexity of a map to study topological properties of kinematic maps in robotics. Because his invariant was not a homotopy invariant, a different, homotopy-invariant approach was subsequently developed by Murillo-Wu and Scott \cite{Murillo_wu, Scott}. Rudyak and Sarkar \cite{Rudyak_soumen} formulated a related definition for the higher topological complexity of a map.  
\begin{definition}\label{defn: higher tc of map}
Let $f\colon X\to Y$ be a continuous map. Let the path space fibration $\pi_n\colon Y^I\to Y^n$ be defined by $$\pi_n(\gamma)=\bigg(\gamma(0),\gamma\bigg(\frac{1}{n-1}\bigg), \dots, \gamma\bigg(\frac{n-2}{n-1}\bigg),\gamma(1)\bigg),$$ where $\gamma \in Y^I$ is a path.
The $n$-th (higher) topological complexity of a map $f$ is defined by $$\TC_n(f)=\secat{f^n}{\pi_n},$$ where $f^n\colon X^n\to Y^n$ is the map defined by $n$-fold Cartesian product of the map $f.$
\end{definition}
 This notion generalizes the higher topological complexity of a space \cite{Rudyak2010, Rudyak_Gon_Higher} and the topological complexity of a map \cite{Scott, Murillo_wu}. Furthermore, Gonz\'alez and Zapata 
\cite[Proposition 3.10]{zapata2023higher} established equivalence between the different versions of the higher topological complexity of maps.

\subsubsection{Algebraic characterization}
We now recall a Whitehead-type characterization of the sectional category. Through this characterization, we make the transition from spaces to algebras.

Let $j \colon X \to Y$ be a cofibration. The $m$-th \emph{fat wedge} of $j$, denoted by $T^m(j)$, is the subspace of $Y^{m+1}$ consisting of all points $(y_0,\dots,y_m)$ such that $y_i \in j(X)$ for at least one $i \in \{0,1,\dots,m\}$. Let $$W_m(j) \colon T^m(j) \hookrightarrow Y^{m+1}$$ denote the inclusion. The notion of the fat wedge can be generalized to an arbitrary map $p\colon E \to B$; see \cite[Definition~6]{Cal_Vandembroucq}.

\begin{theorem}[\cite{FassoVelenik2002, Cal_Vandembroucq}]
Let $p\colon E \to B$ be a map. Then $\secat{}{p}$ is the least non-negative integer $m\leq \infty$ such that the diagonal map $$\Delta_{m+1} \colon B \longrightarrow B^{m+1}$$ factors through the $m$-th fat wedge $T^m(p)$ up to homotopy. That is, there exists a map $s \colon B \to T^m(p)$ such that the diagram
\[
\begin{tikzcd}
	& {T^m(p)} \\
	B & {B^{m+1}}
	\arrow["{W_m(p)}", from=1-2, to=2-2]
	\arrow["s", dashed, from=2-1, to=1-2]
	\arrow["{\Delta_{m+1}}"', from=2-1, to=2-2]
\end{tikzcd}
\]
is homotopy commutative, that is, $W_m(p)\circ s \simeq \Delta_{m+1}$.
\end{theorem}

By \cite[Theorem~1]{FelixTanre2009}, if $\psi \colon \mathcal B \to \mathcal E$ is a surjective CDGA model of a map $p_0 \colon E_0 \to B_0$, then the fat wedge inclusion $$W_m(p_0) \colon T^m(p_0) \to B_0^{m+1}$$ is modeled by the quotient morphism $$q_m \colon \mathcal B^{\otimes (m+1)} \longrightarrow 
\mathcal B^{\otimes (m+1)} / \mathcal N^{\otimes (m+1)},$$ where $\mathcal N = \ker \psi$. Moreover, the multiplication map $$\mu_{m+1} \colon \mathcal B^{\otimes (m+1)} \longrightarrow \mathcal B$$ given by $\mu_{m+1}(b_1\otimes b_2 \otimes \dots\otimes b_{m+1})=b_1b_2\cdots b_{m+1}$,
models the diagonal map $\Delta_{m+1}$. This leads to the following algebraic notion of the sectional category for CDGA morphisms, introduced in \cite[Definition~4]{Carrasquel2015}.
\begin{definition}
Let $\psi \colon \mathcal B \to \mathcal E$ be a surjective CDGA morphism. The sectional category of $\psi$, denoted by $\secat{}{\psi}$, is the smallest non-negative integer $m\leq \infty$ such that there exists a CDGA morphism $\tau$ making the diagram
\[
\begin{tikzcd}
	{\mathcal B^{\otimes (m+1)}} 
	& {\mathcal B^{\otimes (m+1)} / \mathcal N^{\otimes (m+1)}} \\
	{\mathcal B} 
	& {\mathcal B^{\otimes (m+1)} \otimes \Lambda W_m}
	\arrow["{q_m}", from=1-1, to=1-2]
	\arrow["{\mu_{m+1}}"', from=1-1, to=2-1]
	\arrow["{i_m}"', hook, from=1-1, to=2-2]
	\arrow["\sim"', from=2-2, to=1-2]
	\arrow["\tau", dashed, from=2-2, to=2-1]
\end{tikzcd}
\]
commute, where $i_m \colon \mathcal B^{\otimes (m+1)} \to \mathcal B^{\otimes (m+1)} \otimes \Lambda W_m$ is the relative Sullivan model of $q_m$.
\end{definition}

It is shown in \cite[Theorem~7]{Carrasquel2015} that this algebraic definition mirrors the topological notion of the sectional category.

\begin{theorem}
If $\psi$ is a surjective CDGA model for a map $f$, then $$\secat{}{f_0} = \secat{}{\psi}.$$
\end{theorem}

\section{Rational relative sectional category}

The Whitehead-type characterization of the sectional category extends to the relative setting, as shown in \cite[Proposition~4.4]{Calc24}.

\begin{theorem}
Let $K \xrightarrow{f} B \xleftarrow{p} E$ be maps. Then $\secat{f}{p}$ is the least integer $m$ such that the composite $$\Delta_{m+1} \circ f \colon K \longrightarrow B^{m+1}$$ factors through the $m$-th fat wedge $T^m(p)$ up to homotopy. That is, there exists a map $$\gamma \colon K \longrightarrow T^m(p)$$ such that the diagram
\begin{equation}\label{fig: relative secat}
\begin{tikzcd}
	&& {T^m(p)} \\
	K & B & {B^{m+1}}
	\arrow["{W_m(p)}", from=1-3, to=2-3]
	\arrow["\gamma ", dashed, from=2-1, to=1-3]
	\arrow["f"', from=2-1, to=2-2]
	\arrow["{\Delta_{m+1}}"', from=2-2, to=2-3]
\end{tikzcd}
\end{equation}
is homotopy commutative, that is, $$W_m(p) \circ \gamma  \simeq \Delta_{m+1} \circ f.$$
\end{theorem}

Recall that if $\psi \colon \mathcal B \to \mathcal E$ is a surjective CDGA model of a map $p_0 \colon E_0 \to B_0$, then the fat wedge inclusion $$W_m(p_0) \colon T^m(p_0) \longrightarrow B_0^{m+1}$$ is modeled by the quotient morphism $$q_m \colon \mathcal B^{\otimes (m+1)} \longrightarrow 
\mathcal B^{\otimes (m+1)} / \mathcal N^{\otimes (m+1)},$$ where $\mathcal N = \ker \psi$. Moreover, the multiplication map $$\mu_{m+1} \colon \mathcal B^{\otimes (m+1)} \longrightarrow \mathcal B$$ models the diagonal map $\Delta_{m+1}$.
We now provide an algebraic definition of the relative sectional category which will be used to study the relative sectional category in the rational context.
\begin{definition}
Let $\varphi\colon \mathcal B \xrightarrow{} \mathcal K$, $\psi\colon\mathcal B \xrightarrow{} \mathcal E$ be CDGA morphisms with $\psi$ surjective. The \emph{$\varphi$-sectional category of $\psi$}, denoted by $\secat{\varphi}{\psi}$, is the smallest non-negative integer $m$ (or infinity)
such that there exists a CDGA morphism $\tau$ making the diagram
\begin{equation}\label{fig: algebraic relsecat}
\begin{tikzcd}
	{\mathcal B^{\otimes (m+1)}} 
	& {\mathcal B^{\otimes (m+1)} / \mathcal N^{\otimes (m+1)}} \\
	{\mathcal B} 
	& {\mathcal B^{\otimes (m+1)} \otimes \Lambda W_m} \\
	{\mathcal K}
	\arrow["{q_m}", from=1-1, to=1-2]
	\arrow["{\mu_{m+1}}"', from=1-1, to=2-1]
	\arrow["{i_m}"', hook, from=1-1, to=2-2]
	\arrow["\varphi"', from=2-1, to=3-1]
	\arrow["\sim"', from=2-2, to=1-2]
	\arrow["\tau", dashed, from=2-2, to=3-1]
\end{tikzcd}
\end{equation}
commute, where $i_m \colon \mathcal B^{\otimes (m+1)} \to \mathcal B^{\otimes (m+1)} \otimes \Lambda W_m$ is the relative Sullivan model of $q_m$ and $\mathcal N = \ker \psi$.
\end{definition}

The cohomology rings provide various lower bounds for the sectional category, explicitly given by the nilpotency index. We now recall these bounds and examine the relationships among them.
\begin{definition}
For an ideal $I$ in a graded algebra, the \emph{nilpotency} of $I$, denoted by $\nil(I)$, is the smallest non-negative integer $m$ (or infinity) such that $I^{m+1} = 0$. 
\end{definition}
Consider the pullback square associated to maps $f$ and $p$ of topological spaces:
\[
\begin{tikzcd}
	{f^\ast(E)} & E \\
	K & B
	\arrow["g", from=1-1, to=1-2]
	\arrow["q"', from=1-1, to=2-1]
	\arrow["p", from=1-2, to=2-2]
	\arrow["f"', from=2-1, to=2-2]
\end{tikzcd}
\]
If $p$ is a fibration, then $$
\secat{f}{p} = \secat{}{q}.
$$ In this case, the classical cohomological lower bound for sectional category, as shown in \cite{Schwarz1966}, is given by $$
\nil(\ker q^\ast)\leq \secat{}{q}.
$$ In many situations, this bound is sharp. On the other hand, for the relative sectional category, another lower bound
\begin{equation}\label{eq: lb for relsecat}
    \nil\bigl(f^\ast(\ker p^\ast)\bigr) \leq\secat{f}{p}
\end{equation}
was obtained in \cite[Proposition~3.8(5)]{gonzalez2019hopf} by considering the maps $f$ and $p$ separately. We now compare these two bounds.

\begin{prop}
$$\nil\bigl(f^\ast(\ker p^\ast)\bigr) \leq \nil(\ker q^\ast).$$
\end{prop}

\begin{proof}
Let $\alpha \in \ker p^\ast$, so that $p^\ast(\alpha)=0$. Then $$q^\ast f^\ast(\alpha) = g^\ast p^\ast(\alpha) = 0,$$ which implies that $f^\ast(\alpha) \in \ker q^\ast$. Hence, $$
f^\ast(\ker p^\ast) \subseteq \ker q^\ast,$$ and therefore $$\nil\bigl(f^\ast(\ker p^\ast)\bigr) \leq \nil(\ker q^\ast).$$
\end{proof}

We now provide an example showing that the inequality above can be strict. Note that to illustrate this general topological phenomenon, we temporarily work with $\mathbb{Z}_2$ coefficients rather than $\mathbb{Q}$.

\begin{example}\label{ex: contrast to rational version}
Consider the following pullback square:
\[
\begin{tikzcd}
	{f^\ast(P_1(\mathbb{R}P^2))} & {P_1(\mathbb{R}P^2)} \\
	{S^2} & {\mathbb{R}P^2}
	\arrow["g", from=1-1, to=1-2]
	\arrow["q"', from=1-1, to=2-1]
	\arrow["p", from=1-2, to=2-2]
	\arrow["f"', from=2-1, to=2-2]
\end{tikzcd}
\]
Here $f \colon S^2 \to \mathbb{R}P^2$ is the double covering map, and $$P_1(\mathbb{R}P^2)=\{\gamma \in (\mathbb{R}P^2)^I \mid \gamma(1)=\ast\}$$ is the based path space, with $p(\gamma)=\gamma(0)$. Observe that this diagram represents the homotopy pullback of the constant map along $f$. Since $f$ is a fibration, the space $f^\ast(P_1(\mathbb{R}P^2))$, which is the homotopy fiber of $f$, is homotopy equivalent to the fiber of $f$, namely $S^0$.

The induced map in cohomology $$f^\ast \colon H^\ast(\mathbb{R}P^2;\mathbb{Z}_2) \longrightarrow H^\ast(S^2;\mathbb{Z}_2)$$ can be identified with the zero map in positive degrees $$\mathbb{Z}_2[a]/(a^3) \xrightarrow[]{a\mapsto0} \mathbb{Z}_2[x]/(x^2),$$ where $|a|=1$ and $|x|=2$. Moreover, the maps $p^\ast$ and $q^\ast$ are trivial in positive degrees, 
since the domains of both $p$ and $q$ are homotopy equivalent to finite discrete spaces.

Consequently, $f^\ast(\ker p^\ast)=0,$ and hence $\nil\bigl(f^\ast(\ker p^\ast)\bigr)=0.
$ On the other hand, $\ker q^\ast = (x)$, and therefore $\nil(\ker q^\ast)=1.$

This shows that $$\nil\bigl(f^\ast(\ker p^\ast)\bigr) < \nil(\ker q^\ast)$$ for this example. In particular, we have $\nil\bigl(f^\ast(\ker p^\ast)\bigr)<\secat{f}{p}$.
\end{example}

We establish a relation between the relative sectional category at the space level and at the CDGA level. We extend the ideas of \cite[Theorem 7]{Carrasquel2015}. The presence of the new map $f$ introduces additional complexity to the proof.
\begin{theorem}\label{thm: secat top vs algebra}
    Let $f\colon K\to B$ and $p\colon E\to B$ be two maps, and let $f_0$ and $p_0$ denote their rationalizations respectively.
    Let $\psi\colon\c B\to \c E$ be a surjective model for $p_0$ and $\varphi\colon \c B\to \c K$ be any model for $f_0$.
     Then $$\mathrm{secat}_{f_0}(p_0)=\mathrm{secat}_{\varphi}(\psi).$$
\end{theorem}
\begin{proof}
Let $\mathcal N = \ker \psi$. As mentioned earlier, the quotient map $$q_m \colon \mathcal B^{\otimes (m+1)} \longrightarrow 
\mathcal B^{\otimes (m+1)} / \mathcal N^{\otimes (m+1)},$$ models the map 
$W_m(p_0) \colon T^m(p_0) \to B_0^{m+1}$.

Let $\alpha\colon \Lambda V\xra{\sim} \c B$ be the minimal model of $\c B$ and $\xi\colon\Lambda V\xra{\sim} A_{PL}( B_0^{m+1})$ be the minimal model of $B_0^{m+1}$. Let $i\colon\Lambda V^{\otimes m+1}\to \Lambda V^{\otimes m+1}\otimes\Lambda W_m$ be the relative Sullivan model of $q_m\circ(\otimes \alpha).$ Then, we have a commutative diagram
\[\begin{tikzcd}[row sep=1.3cm, column sep=2.0cm]
{\mathcal{B}^{\otimes m+1}} \arrow[r, "{q_m}"] & {\frac{\mathcal{B}^{\otimes m+1}}{\mathcal{N}^{\otimes m+1}}} \\
{\Lambda V^{\otimes m+1}} \arrow[u, "{\otimes \alpha}"] \arrow[r, "i"] \arrow[d, "\xi"'] & {\Lambda V^{\otimes m+1}\otimes \Lambda W_m} \arrow[u, "\sim"'] \arrow[d, "\rho"] \\
{A_{PL}(B_0^{ m+1})} \arrow[r, "{A_{PL}{(W_m(p_0))}}"'] & {A_{PL}(T^m(p_0))}
\end{tikzcd}\]
where $\rho$ is a quasi-isomorphism.

Let $\secat{f_0}{p_0}=m$. Then, there exists $\gamma\colon K_0\to T^m(p_0)$ such that the following diagram becomes commutative as in  \eqref{fig: relative secat} 
\begin{equation}\label{fig: rational relative secat}
\begin{tikzcd}
	&& {T^m(p_0)} \\
	K_0 & B_0 & {B_0^{m+1}}
	\arrow["{W_m(p_0)}", from=1-3, to=2-3]
	\arrow["\gamma ", dashed, from=2-1, to=1-3]
	\arrow["f_0"', from=2-1, to=2-2]
	\arrow["{\Delta_{m+1}}"', from=2-2, to=2-3]
\end{tikzcd}
\end{equation}
Let $\tilde \varphi\colon\Lambda V\to \Lambda U$ be a model for $f_0\colon K_0\to  B_0$ 
where $\Lambda U$ is minimal, then we have a homotopy commutative diagram
\begin{equation}\label{fig: Model of B,K}
\begin{tikzcd}
	{A_{PL}(B_0)} & \Lambda V & {\c B} \\
	A_{PL}(K_0) & \Lambda U & \c K
	\arrow[ from=1-1, to=2-1, "A_{PL}(f_0)"']
	\arrow["\alpha'"', from=1-2, to=1-1]
	\arrow["\alpha", from=1-2, to=1-3]
	\arrow["\tilde \varphi", from=1-2, to=2-2]
	\arrow["\varphi", from=1-3, to=2-3]
	\arrow[ from=2-2, to=2-1, "\beta'"]
	\arrow[ "\beta"', from=2-2, to=2-3]
\end{tikzcd}    
\end{equation}
where $\alpha', \beta, \beta'$ are quasi-isomorphisms  giving the minimal models of $\c K$ and $K_0$ respectively.

Using the left homotopy commutative square of diagram \eqref{fig: Model of B,K} and applying $A_{PL}$ to diagram \eqref{fig: rational relative secat}, we obtain the following homotopy commutative diagram:
\[\begin{tikzcd}[row sep=1.5cm, column sep=1.8cm]
{} & {} & {\Lambda V} \arrow[r, "{\tilde \varphi}"] \arrow[d, "{\alpha'}"] & {\Lambda U} \arrow[d, "{\beta'}"] \\
{\Lambda V^{\otimes m+1}} \arrow[r, "\xi"] \arrow[d, "i"', hook] \arrow[urr, "{\mu'_{m+1}}", bend left=15] \arrow[drrr, "{A_{PL}(W_mp_0)\circ \otimes \alpha'}"', sloped, pos=0.25] & {A_{PL}(B_0^{ m+1})} \arrow[r, "{A_{PL}(\Delta_{m+1})}"] & {A_{PL}(B_0)} \arrow[r, "{A_{PL}(f_0)}"] & {A_{PL}(K_0)} \\
{\Lambda V^{\otimes m+1}\otimes \Lambda W_m} \arrow[rrr, "\rho"'] & {} & {} & {A_{PL}(T^mp_0)} \arrow[u, "{A_{PL}(\gamma)}"']
\end{tikzcd}\]
where $\mu_{m+1}'$ denotes the multiplication map whose domain is $\Lambda V^{\otimes m+1}$.

 From the above diagram, we extract the following homotopy commutative square:
\begin{equation}
\begin{tikzcd}
	{\Lambda V^{\otimes m+1} } & {\Lambda U} \\
	{\Lambda V^{\otimes m+1} \otimes \Lambda W} & {A_{PL}(K_0)}
	\arrow["{\tilde \varphi \mu_{m+1}'}",  from=1-1, to=1-2]
	\arrow["i"', hook, from=1-1, to=2-1]
	\arrow["\sim", from=1-2, to=2-2]
	\arrow["{A_{PL}(\gamma)\rho}"', from=2-1, to=2-2].
\end{tikzcd}    
\end{equation}

 Using Proposition~\ref{prop:lifting prop for relative sullivan}, we get a lift $\tilde \gamma \colon\Lambda V^{\otimes m+1}\otimes \Lambda W_m\to \Lambda U$ such that 
 \begin{equation}\label{eq: auxiliary lift}
     \tilde \varphi \mu'_{m+1}=\tilde \gamma i.
 \end{equation}
 Finally, we have a commutative diagram where $i_m$ is a relative Sullivan algebra given by the pushout of $i$ along $\otimes \alpha$. Since $i$ is a cofibration, the inner square is also a homotopy pushout diagram.
\[\begin{tikzcd}
	{\Lambda V^{\otimes m+1} } & {\mathcal B^{\otimes m+1}} \\
	{\Lambda V^{\otimes m+1} \otimes \Lambda W_m} & {\mathcal B^{\otimes m+1} \otimes \Lambda W_m} \\
	&& {\mathcal K}
	\arrow["{\otimes \alpha}", from=1-1, to=1-2]
	\arrow["i"', hook, from=1-1, to=2-1]
	\arrow["{i_m}", hook, from=1-2, to=2-2]
	\arrow["{\varphi \mu_{m+1}}",  bend left=25, from=1-2, to=3-3]
	\arrow[from=2-1, to=2-2]
	\arrow["{\beta\tilde\gamma}"',  bend right=10, from=2-1, to=3-3]
	\arrow["{\tau'}"', dashed, from=2-2, to=3-3]
\end{tikzcd}\]
 We note that the outer square of the above 
 diagram is homotopy commutative as 
 \begin{align*}
     \varphi \mu_{m+1}\circ(\otimes\alpha)&=\varphi\alpha\mu'_{m+1}\\
     &\simeq \beta\tilde \varphi\mu'_{m+1} \quad
     \text{using diagram \eqref{fig: Model of B,K}}\\
     &=\beta\tilde \gamma i ~~~~\qquad \text{using \eqref{eq: auxiliary lift}}.
 \end{align*}
 By the universal property, we get $\tau'$ such that $\tau'i_m\simeq \varphi\mu_{m+1}.$ By Lemma~\ref{lem: make homotopy to equality},
 we get a map 
 $\tau$ such that $\tau i_m=\varphi\mu_{m+1}.$ Therefore, by definition, $\secat{\varphi}{\psi}\leq m$; that is,  $\secat{\varphi}{\psi}\leq \secat{f_0}{p_0}.$
 \par
 The converse follows from the application of the spatial realization functor.
\end{proof}
\begin{remark}Taking the pushout of $\varphi\mu_{m+1}$ along $i_m$, we get a commutative diagram
\[\begin{tikzcd}
	{\mathcal B^{\otimes m+1}} && {\mathcal B^{\otimes m+1}\otimes \Lambda W_m} & \\
	{\mathcal K} && {\mathcal K\otimes \Lambda W_m} \\
	&&& {\mathcal K}
	\arrow["{i_m}", hook, from=1-1, to=1-3]
	\arrow["{\varphi\mu_{m+1}}"', from=1-1, to=2-1]
	\arrow["\gamma", from=1-3, to=2-3]
	\arrow["\tau", bend left=20, from=1-3, to=3-4]
	\arrow["{j_m}", from=2-1, to=2-3]
	\arrow["\id", bend right=10, from=2-1, to=3-4]
	\arrow["\sigma"', dashed, from=2-3, to=3-4]
\end{tikzcd}\]

    If $\tau i_m=\varphi \mu_{m+1}$, then by the universal property, we get a map $\sigma$ such that $\sigma j_m=\id$.
    
    Conversely, assume that we have a map $\sigma$ such that $\sigma j_m=\id$. This gives a map $\tau=\sigma\gamma$. Then $\tau i_m=\sigma\gamma i_m=\sigma j_m\varphi\mu_{m+1}.$
    Therefore, we can redefine $\secat{\varphi}{\psi}$ to be the minimum $m$ such that $j_m$ admits a retract.
\end{remark}
Following \cite{Carrasquel2015}, we introduce several auxiliary notions that will be useful for establishing subsequent inequalities.
\begin{definition}
    Given a pair of CDGA maps $\c K\xla{\varphi}\c B\xra{\psi}\c E$ where $\psi$ is surjective. Then we define \begin{enumerate}
        \item 
    $\mathrm{m}\secat{\varphi}{\psi}=\min\{m\mid j_m \text{ admits a }\c B\text{-module retract}\}$.
    \item H$\secat{\varphi}{\psi}=\min\{m\mid j_m^\ast \text{ is injective}\}$.  \end{enumerate}
\end{definition}
Now, let us observe the relation between $\secat{\varphi}{\psi},$ $\mathrm{m}\secat{\varphi}{\psi}$, $\mathrm{H}\secat{\varphi}{\psi},$ and $\nil({\varphi}^\ast(\ker{\psi}^\ast)).$
\begin{prop} \label{prop: compare nilker vs secat}
For a pair of CDGA maps $\varphi$ and $\psi$ with $\psi$ surjective, there is a chain of inequalities: $$\nil({\varphi}^\ast(\ker{\psi}^\ast))\leq \H\secat{\varphi}{\psi}\leq\mathrm{m}\secat{\varphi}{\psi}\leq \secat{\varphi}{\psi}.$$
\end{prop}
\begin{proof}
    The third inequality is immediate because if a CDGA map $\sigma$ exists, it is, in particular, a $\mathcal{B}$-module map.
    
    For the second inequality, if $\sigma j_m=\id$ holds for a $\c B$-module map $\sigma$, then passing to cohomology yields $\sigma^\ast j_m^\ast=\id$. Hence, $j_m^\ast$ is injective.
    
    For the first inequality, let us assume $\H\secat{\varphi}{\psi}=m.$ So, $j_m^\ast$ is injective. We want to show that for any $(m+1)$ cohomological classes $[x_1],\dots, [x_{m+1}]\in \ker\psi^\ast$, we have $\varphi^\ast([x_1]\dotsi[x_{m+1}])=0.$
    
    Since $\psi^\ast([x_k])=0$ for all $k=1,\dots,m+1,$ we have $\psi(x_k)=de_k$ for some $e_k\in\c E$. Since $\psi$ is surjective, $e_k=\psi (b_k)$ for some $b_k\in \c B.$ Therefore, $x_k-db_k\in \ker\psi=\c N.$

Now, let $z=(x_1-db_1)\otimes\dotsi\otimes(x_{m+1}-db_{m+1})\in \c N^{\otimes m+1}$.
We have \begin{align*}
    j_m^\ast\varphi^\ast([x_1]\dotsi[x_{m+1}])&=j_m^\ast\varphi^\ast([x_1-db_1]\dotsi[x_{m+1}-db_{m+1}])\\ &=j_m^\ast\varphi^\ast\mu_{m+1}^\ast[z]\\ &=\gamma^\ast i_m^\ast[z]=\gamma^\ast(0)=0.
\end{align*}
Since $j_m^\ast$ is injective, we have $\varphi^\ast([x_1]\dotsi[x_{m+1}])=0.$ Thus, $\nil({\varphi}^\ast(\ker{\psi}^\ast))\leq \H\secat{\varphi}{\psi}$.
\end{proof}

 Let $\c K\xla{\varphi} \c B\xra{\psi} \c E$ be CDGA maps where $\psi$ is surjective. Let $\c N=\ker \psi$. Then we have the following commutative diagram:
\[\begin{tikzcd}
	{\mathcal B} & {\mathcal B/\mathcal N^{{m+1}}} \\
	& {\mathcal B\otimes \Lambda Z_m} \\
	{\mathcal K} & {\mathcal K \otimes \Lambda Z_m}
	\arrow["{\underline q_m}", two heads, from=1-1, to=1-2]
	\arrow["{\underline i_m}"', hook, from=1-1, to=2-2]
	\arrow["\varphi"', from=1-1, to=3-1]
	\arrow["\sim"', two heads, from=2-2, to=1-2]
	\arrow["{\underline \nu}", from=2-2, to=3-2]
	\arrow["{\underline j_m}"', from=3-1, to=3-2]
\end{tikzcd}\]
where $\underline q_m$ is the quotient map, $\underline{i}_m$ denotes the relative Sullivan model for $\u q_m$. The lower square is obtained by taking the pushout of $\varphi$ along $\underline{i}_m.$

We now introduce some additional natural notions which will be used to establish some relations. 
 \begin{definition}
 \begin{enumerate}
     \item 
     $\sc{\varphi}{\psi}$ is the minimum $m$ such that there is a CDGA map $\underline{\sigma}\colon\c K\otimes \Lambda Z_m\to \c K$ satisfying $\underline{\sigma}\circ\underline{j}_m=\id$.
     \item 
     $\msc{\varphi}{\psi}$ is the minimum $m$ such that there is a $\c B$-module map $\underline{\sigma}\colon\c K\otimes \Lambda Z_m\to \c K$ satisfying $\underline{\sigma}\circ\underline{j}_m=\id$.
     \item 
     $\hsc{\varphi}{\psi}$ is the minimum $m$ such that $\underline{j}_m^\ast$ is injective.
    \end{enumerate}
 \end{definition}

 It is worth noting that these new notions serve as upper bounds for different versions of algebraic sectional categories.
\begin{prop}\label{prop: compare secat vs sc}
For a pair of CDGA maps $\varphi$ and $\psi$ with $\psi$ surjective, we have the following inequalities:
\begin{enumerate}
    \item  $\secat{\varphi}{\psi}\leq\sc{\varphi}{\psi},$
    \item  $\mathrm{m}\secat{\varphi}{\psi}\leq\msc{\varphi}{\psi},$ 
    \item  $\mathrm{H}\secat{\varphi}{\psi}\leq\hsc{\varphi}{\psi}$.
\end{enumerate}
\end{prop}
\begin{proof}
     We have the following diagram of commutative squares,
\[\begin{tikzcd}
	{\mathcal B^{\otimes m+1}} & {\mathcal B} & {\mathcal K} & {\mathcal K} \\
	& {\mathcal B \otimes \Lambda Z_m} & {\mathcal K \otimes \Lambda Z_m} \\
	{\mathcal B^{\otimes m+1} \otimes \Lambda W_m} &&& {\mathcal K \otimes \Lambda W_m}
	\arrow["{ \mu_{m+1}}", from=1-1, to=1-2]
	\arrow["{i_m}"', hook, from=1-1, to=3-1]
	\arrow["\varphi", from=1-2, to=1-3]
	\arrow["{\underline i_m}", hook, from=1-2, to=2-2]
	\arrow["\id", from=1-3, to=1-4]
	\arrow["{\underline j_m}", hook, from=1-3, to=2-3]
	\arrow["{j_m}", hook, from=1-4, to=3-4]
	\arrow["{\underline \nu}"', from=2-2, to=2-3]
	\arrow[from=3-1, to=3-4]
\end{tikzcd}\]
where the inner and outer squares are the pushout squares.
    
    Let $\underline{\mu}_{m+1}\colon \mathcal B^{\otimes m+1}/\mathcal N^{\otimes m+1}  \to \mathcal B/\mathcal N^{m+1}$ be induced by the multiplication map $\mu_{m+1}\colon\c B^{\otimes{m+1}}\to \c B.$ Then we get $\underline{\mu}_{m+1}\circ q_m=\underline{q}_m\circ \mu_{m+1}.$ 
    Using \cite[Proposition~14.6]{FelixHalperinThomas2001}, we have a lift $\alpha$
\[\begin{tikzcd}
	{\mathcal B^{\otimes m+1}} & {\mathcal B} & {\mathcal B \otimes \Lambda Z_m} \\
	{\mathcal B^{\otimes m+1} \otimes \Lambda W_m} & {\mathcal B^{\otimes m+1}/\mathcal N^{\otimes m+1} } & {\mathcal B/\mathcal N^{m+1}}
	\arrow["{ \mu_{m+1}}", from=1-1, to=1-2]
	\arrow["{i_m}"', hook, from=1-1, to=2-1]
	\arrow["{\underline i_m}", from=1-2, to=1-3]
	\arrow["\sim", two heads, from=1-3, to=2-3]
	\arrow["\alpha"', dashed, from=2-1, to=1-3]
	\arrow["\sim"', from=2-1, to=2-2]
	\arrow["{\underline \mu_{m+1}}"', from=2-2, to=2-3]
\end{tikzcd}\]
 which makes the appropriate triangles commute.
    In particular, we get $\alpha\circ i_m=\underline{i}_m\circ \mu_{m+1}$. Since $$\underline{\nu}\circ\alpha\circ  i_m=\underline{\nu}\circ\underline{i}_m\circ\mu_{m+1}=\underline{j}_m\circ\varphi\circ\mu_{m+1},$$ by the universal property of the pushout square,
\[
\begin{tikzcd}
	{\mathcal B^{\otimes m+1}} & {\mathcal K} \\
	{\mathcal B^{\otimes m+1} \otimes \Lambda W_m} & {\mathcal K\otimes \Lambda W_m} \\
	&& {\mathcal K\otimes \Lambda Z_m}
	\arrow["{\varphi \mu_{m+1}}", from=1-1, to=1-2]
	\arrow["{i_m}"', hook, from=1-1, to=2-1]
	\arrow["{j_m}", hook, from=1-2, to=2-2]
	\arrow["{\underline j_m}", bend left=25, from=1-2, to=3-3]
	\arrow[from=2-1, to=2-2]
	\arrow["{\underline \nu \alpha}"', bend right=15, from=2-1, to=3-3]
	\arrow["\rho"', dashed, from=2-2, to=3-3]
\end{tikzcd}
\]
we have a map $\rho$ that makes the triangles commute.

Let $\sc{\varphi}{\psi}=m.$
    So, there exists a CDGA map $\underline{\sigma}$ such that $\underline{\sigma}\circ \underline{j}_m=\id.$
    It yields $$\underline{\sigma}\circ\rho\circ j_m=\underline{\sigma}\circ\underline{j}_m=\id.$$ Therefore, $\secat{\varphi}{\psi}\leq m.$
  
    Similarly, we can prove the second inequality.
  For the third inequality, we note that $\rho\circ j_m=\underline{j}_m$. Passing to cohomology yields, $\rho^\ast\circ j_m^\ast=\underline{j}_m^\ast.$ Therefore, injectivity of $\underline{j}_m^\ast$ implies that $j^\ast_m$ is also injective. 
\end{proof}

 We now compare these algebraic notions.
\begin{prop} \label{prop: compare Hsc msc sc nil}
For a pair of CDGA maps $\varphi$ and $\psi$ with $\psi$ surjective, there is a chain of inequalities: $$\hsc{\varphi}{\psi}\leq \msc{\varphi}{\psi}\leq \sc{\varphi}{\psi}\leq \nil{(\varphi(\ker\psi))}.$$
\end{prop}
\begin{proof}
    Let us prove the last inequality. Assume  $\nil{(\varphi(\ker\psi))}=m$. Then $\varphi(\c N)^{m+1}=0$, which implies $\varphi(\c N^{m+1})=0.$
    \par
    Then we have the commutative diagram obtained by taking pushouts of $\varphi$ along $\underline{i}_m$ and then along $\underline{q}_m\colon\mathcal B\to \mathcal B/\mathcal N^{m+1}$.
\[\begin{tikzcd}
	{\mathcal B} & {\mathcal B \otimes \Lambda Z_m} & {\mathcal B/\mathcal N^{{m+1}}} \\
	{\mathcal K} & {\mathcal K \otimes \Lambda Z_m} \\
	{\mathcal K} & {} & {\mathcal K/\varphi(\mathcal N^{{m+1})}}
	\arrow["{\underline i_m}", from=1-1, to=1-2]
	\arrow["\varphi"', from=1-1, to=2-1]
	\arrow["\sim", from=1-2, to=1-3]
	\arrow[from=1-3, to=3-3]
	\arrow["{\underline j_m}"', hook, from=2-1, to=2-2]
	\arrow["\id"', from=2-1, to=3-1]
	\arrow["\sim" ', two heads, from=1-2, to=2-2]
	\arrow["{\underline \lambda}", dashed, from=2-2, to=3-3]
	\arrow["q'_m"', from=3-1, to=3-3]
\end{tikzcd}\]
    By the universal property of the inner pushout square, we obtain a CDGA map $$\underline \lambda\colon\c K\otimes\Lambda Z_m\to \frac{\c K}{\varphi (\c N^{m+1})}$$ such that $\underline \lambda\circ\underline{j}_m=q'_m.$
    Now $\varphi(\c N^{m+1})=0,$ so $\underline \lambda\colon\c K\otimes \Lambda Z_m\to\c K$ satisfies 
    $\underline \lambda \circ \underline{j}_m = \id$.
    Therefore, $\sc{\varphi}{\psi}\leq m.$
    
    The other inequalities follow easily.
\end{proof}
Finally, we arrive at the main result. In Example~\ref{ex: contrast to rational version}, we observed a general topological setting where the nilpotency lower bound is strictly less than the relative sectional category (i.e., $\nil\bigl(f^\ast(\ker p^\ast)\bigr) < \secat{f}{p}$). In 
contrast to this strict inequality, Theorem~\ref{thm: secat for formal cdga map} below shows that for formal CDGA maps, the algebraic relative sectional category coincides with its nilpotency lower bound. Consequently, the rational topological complexity of formal maps agrees with its cohomological lower bound (generalizing \cite[Theorem~24]{Carrasquel2015}, as we show in Proposition~\ref{prop: Tc for formal map}).

\begin{theorem}\label{thm: secat for formal cdga map}
    Let $\varphi$ and $\psi$ be formal CDGA maps such that $\psi$ 
    and $\psi^\ast$ are surjective. Then $$\secat{\varphi}{\psi}= \nil({\varphi^\ast}(\ker{\psi^\ast})).$$
\end{theorem}
\begin{proof}
    Since $\varphi$ and $\psi$ are formal, $\secat{\varphi}{\psi}=\secat{\varphi^\ast}{\psi^\ast}.$ By Proposition~\ref{prop: compare secat vs sc} and Proposition~\ref{prop: compare Hsc msc sc nil}, we know $$\secat{\varphi^\ast}{\psi^\ast}\leq\sc{\varphi^\ast}{\psi^\ast}\leq \nil({\varphi^\ast}(\ker{\psi^\ast})).$$
Conversely, by Proposition~\ref{prop: compare nilker vs secat}, we have $$\nil({\varphi^\ast}(\ker{\psi^\ast})) \leq \secat{\varphi^\ast} {\psi^\ast}.$$  The equality now follows immediately.
\end{proof}
The preceding theorem yields the following topological consequence. Indeed, by choosing a surjective CDGA model for $p$ and using formality of the maps, we obtain the following.
\begin{theorem}\label{thm: TOPsecat for formal cdga map}
    Let $f\colon K\to B$ and $p\colon E\to B$ be formal maps such that  $p^\ast$ is surjective. Then $$\mathrm{secat}_{f_0}(p_0)= \nil({f^\ast}(\ker{p^\ast})).$$
\end{theorem}
\section{Applications}
We now present several applications of the rational version of relative sectional category.
It is well known that the rational Lusternik-Schnirelmann category of a space admits
an algebraic description in terms of Sullivan models \cite{FelixHalperinThomas2001}.
Moreover, if $X$ is a formal space, then $$\cat(X_0)=\nil(H^{>0}(X;\QQ)),$$ see \cite{FelixHalperin1982,FelixHalperinThomas2001}. We generalize these results to maps between topological spaces.

Following \cite{LScat_map}, we recall that the \emph{Lusternik-Schnirelmann category} $\cat(f)$ of a map $f \colon X \to Y$ is the least non-negative integer $n \leq \infty$ such that $X$ admits a cover by $n+1$ open sets $\{U_i\}_{i=0}^n$ on which the restriction $f|_{U_i}$ is null-homotopic. It is immediate from the definition that $$\cat(f)=\secat{f}{\ast \to Y},$$ that is, $\cat(f)$ coincides with the relative sectional category of the inclusion of a point.

\begin{prop}\label{prop: rational cat of f}
If $\varphi \colon \mathcal Y \to \mathcal X$ is a CDGA model for a map $f \colon X \to Y$, then $$\cat(f_0)=\secat{\varphi}{\epsilon \colon \mathcal Y \to \mathbb Q},$$ where $\epsilon$ denotes the augmentation.
\end{prop}

\begin{proof}
By definition, $\cat(f_0)=\secat{f_0}{\ast \to Y_0}$. The inclusion $\ast \to Y_0$ is modeled by the augmentation $\epsilon \colon \mathcal Y \to \mathbb Q$. The result now follows from Theorem~\ref{thm: secat top vs algebra}.
\end{proof}

\begin{prop}
    For a formal map $f\colon X\to Y$, the LS-category of $f_0$ is given by $$\cat(f_0)=\nil (f^\ast(H^\ast(Y))^+),$$ where the notation $A^+$ denotes the ideal of positive degree elements in the CDGA $A$.
\end{prop}
\begin{proof}
As $Y$ is formal, the inclusion $i\colon\ast \to Y_0$ is formal. Moreover, $i^\ast$ is surjective and  $\ker i^\ast=(H^\ast(Y))^+.$
Therefore, by Theorem~\ref{thm: TOPsecat for formal cdga map}, we conclude $$\cat(f_0)=\nil (f^\ast(\ker \epsilon))=\nil (f^\ast(H^\ast(Y))^+).$$
\end{proof}

The rational topological complexity of a space also admits an algebraic description in terms of
Sullivan models \cite{JessupMurilloParent2012}.
Moreover, if $X$ is a formal space, then \cite{JessupMurilloParent2012} prove $$\TC(X_0)=\nil(\ker \mu),$$ where $\mu\colon H^\ast(X)\otimes H^\ast(X)\to H^\ast(X)$ is the multiplication map. We generalize these results to maps.
We note that, since $\pi_n$ is a fibration replacement of the diagonal map $\Delta_n \colon Y \to Y^n$, the map $\pi_n$ can be replaced by $\Delta_n$ in Definition~\ref{defn: higher tc of map}.
\begin{prop}\label{prop: rational version of tc}
If $\varphi \colon \mathcal Y \to \mathcal X$ is a CDGA model for $f \colon X \to Y$ and $\mu_n \colon \mathcal Y^{\otimes n} \to \mathcal Y$ denotes the $n$-fold multiplication, then $$\TC_n(f_0)=\secat{\varphi^{\otimes n}}{\mu_n}.$$
\end{prop}

\begin{proof}
By definition, $\TC_n(f_0)=\secat{f_0^{ n}}{\Delta_n}$, where $\Delta_n \colon Y_0 \to (Y_0)^n$ is the rationalized diagonal map. Since $\Delta_n$ is modeled by $\mu_n$, the result follows immediately from Theorem~\ref{thm: secat top vs algebra}.
\end{proof}

\begin{prop}\label{prop: Tc for formal map}
For a formal map $f\colon X \to Y$, the $n$-th topological complexity of the map $f_0$ is given by $$\TC_n(f_0)=\nil(({f^\ast})^{\otimes n}(\ker \mu_n)),$$ where $\mu_n\colon  H^\ast (Y)^{\otimes n} \to H^\ast(Y)$ is the multiplication map.
\end{prop}
\begin{proof}
    Since $Y$ is a formal space, the diagonal map $\Delta_n\colon  Y \to Y^n$ is a formal map. 
    We note that the map $\Delta_n\colon Y\to Y^n$ induces the surjective map $\mu_n\colon  H^\ast (Y)^{\otimes n} \to H^\ast(Y)$ on cohomology. 
    Moreover, since $f\colon X \to Y$ is formal, the product map $f^{n}\colon X^n\to Y^n$ is a formal map.
    Therefore, by Theorem~\ref{thm: TOPsecat for formal cdga map}, we conclude $$\TC_n(f_0)=\nil(({f^\ast})^{\otimes n}(\ker \mu_n)).$$
\end{proof}
The following lemma shows that the nilpotency of an ideal is preserved under monomorphisms. This fact will be used repeatedly in our computations.
\begin{lemma}\label{lem: nilpotency preserved by mono}
    Let $g \colon A \to B$ be a monomorphism of graded commutative algebras and let $I \subset A$ be an ideal. Then $$\nil(I)=\nil(g(I)).$$
\end{lemma}
\begin{proof}
Since $g$ is injective, for every $r\geq 1$, we have $$I^r=0\iff g(I^r)=g(I)^r=0.$$ Therefore, $\nil(I) =\nil(g(I)).$

\end{proof}
We now explicitly compute the higher topological complexity of certain maps using Proposition~\ref{prop: Tc for formal map}. 

\begin{example}\label{ex: Tc of sphere map}
Let us consider a map $f \colon S^k \to S^l$ between spheres for $2\leq k,l$ and $k\neq 2l-1$. Then $f$  is formal  by the result in \cite[Corollary 3.2]{{Arkowitz1988}}. The rational cohomology algebra of $S^l$ is $$
H^\ast(S^l)=
\begin{cases}
\Lambda(x), & l \text{ odd},\\
\Lambda(x)/(x^2), & l \text{ even},
\end{cases}
$$ where $|x|=l$.

\noindent\textbf{Case 1.} If $k\neq l$, then $f^\ast=0$ on reduced cohomology for degree reasons. Consequently, $$({f^\ast})^{\otimes n}(\ker \mu_n)=0.$$ Therefore, $\TC_n(f_0)=0$.

\noindent\textbf{Case 2.} We now assume $k=l$.
The kernel of the multiplication map $$\mu_n \colon H^\ast(S^l)^{\otimes n} \to H^\ast(S^l)$$ is generated by the elements $$z_i := x_i - x_1, \qquad 2\le i \le n,$$ where $x_i=1\otimes \cdots \otimes x \otimes \cdots \otimes 1$ denotes $x$ in the $i$-th tensor factor.

\noindent\textbf{Case 2(a).} If $\deg(f)=0$, then $(f^\ast)^{\otimes n}(\ker \mu_n)=0$, and thus $\TC_n(f_0)=0$.

\noindent\textbf{Case 2(b).} If $\deg(f)\neq0$, then $f^\ast$ is an isomorphism. 
\begin{itemize}
    \item 
Assume first that $l$ is odd. Then $x_i x_j = -x_j x_i$ for $i\neq j$ and $x_i^2=0.$  The product of all the generators
\begin{equation*}
z_2 z_3 \cdots z_n
= \sum_{i=1}^n (-1)^{i+1} x_1 x_2 \cdots \widehat{x_i} \cdots x_n \neq 0,  
\end{equation*}
so $(\ker \mu_n)^{n-1}\neq 0$.  
Moreover, we have
\begin{align*}
   z_i^2=(x_i-x_1)^2=0    
\end{align*}
for all $i$, therefore any further repetition of a factor $z_i$ in the product
$z_2 z_3 \cdots z_n$ forces the product to vanish, so we have $(\ker \mu_n)^n=0$. Hence $\nil(\ker \mu_n)=n-1$, and by Lemma~\ref{lem: nilpotency preserved by mono} and Proposition~\ref{prop: Tc for formal map}, $$\TC_n(f_0)=n-1.$$

\item Now suppose that $l$ is even.  Then $x_i x_j = x_j x_i$ for $i\neq j$ and $x_i^2=0.$  We have
\begin{align*}
   z_2^2=(x_2-x_1)^2=-2x_1x_2     
\end{align*}
so that $$z_2^2 z_3 \cdots z_n=-2x_1x_2(x_3-x_1)\cdots(x_n-x_1)=-2x_1x_2\cdots x_n \neq 0.
$$ Thus $(\ker \mu_n)^n\neq 0$. On the other hand, $(\ker \mu_n)^{n+1}=0$, since elements in this ideal have total degree $(n+1)l$, which exceeds the top non-zero degree in $H^\ast(S^l)^{\otimes n}$. Therefore, $\nil(\ker \mu_n)=n$, and again by Lemma~\ref{lem: nilpotency preserved by mono} and Proposition~\ref{prop: Tc for formal map}, $$\TC_n(f_0)=n.$$
\end{itemize}
\end{example}

We refer to \cite{higher_milnor} for similar but more involved calculations used to study the higher topological complexity of Milnor manifolds.
\begin{example}
It is well known that the projection maps $p_{X'}\colon X\times X' \to X'$ are formal provided both $X$ and $X'$ are formal. Let us consider the projection $f\colon \CC P^k\times \CC P^l \to \CC P^l$ onto the second factor, where $k,l\geq 1$.

The rational cohomology algebra of $\CC P^l$ is $$H^\ast(\CC P^l)=\Lambda(x)/(x^{l+1}),$$ where $|x|=2$.

The kernel of the multiplication map $$\mu_n\colon H^\ast(\CC P^l)^{\otimes n}\to H^\ast(\CC P^l)$$ is generated by $$\{w_i:=x_i-x_{i-1}\mid 2\le i\le n\}.$$

A direct binomial computation gives
\begin{align*}
w_i^{2l}
&=(x_i-x_{i-1})^{2l}\\
&=\sum_{j=0}^{2l}\binom{2l}{j}x_i^j(-x_{i-1})^{2l-j}\\
&=\binom{2l}{l}(-1)^l x_{i-1}^l x_i^l,
\end{align*}
since all other terms vanish for degree reasons (as $x_i^{l+1}=0$).
Set $\gamma=(-1)^l\binom{2l}{l}$).
\begin{itemize}
    \item 
If $n=2s$ is even, then $$P:=w_2^{2l}w_4^{2l}\cdots w_{2s}^{2l}
=\gamma^s x_1^l x_2^l\cdots x_{2s-1}^l x_{2s}^l\neq 0,
$$ so $\nil(\ker\mu_n)\ge 2ls=nl$.

\item If $n=2s+1$ is odd, then $$P\cdot w_{2s+1}^l
=\gamma^s x_1^l x_2^l\cdots x_{2s}^l x_{2s+1}^l\neq 0,$$ and again $\nil(\ker\mu_n)\ge 2ls+l=nl$.
\end{itemize}
Since the product of these classes exhausts the top possible non-zero degree in $H^\ast(\CC P^l)^{\otimes n}$, we have $(\ker\mu_n)^{nl+1}=0$, and hence $\nil(\ker\mu_n)=nl$.

Finally, since $f^\ast (x)= 1\otimes x $, the map $f^\ast$ is a monomorphism.
Therefore, by Lemma~\ref{lem: nilpotency preserved by mono} and Proposition~\ref{prop: Tc for formal map}, $$\TC_n(f_0)=nl.$$
\end{example}

\begin{example}
It is well known that the diagonal map $\Delta\colon X \to X\times X$ is formal provided $X$ is formal. Let $X= \CC P^l$ and $Y=  \CC P^l\times \CC P^l$, and consider the diagonal map $f\colon \CC P^l \to \CC P^l\times \CC P^l$ where $l\geq 1$.

The rational cohomology algebra of $\CC P^l$ and $\CC P^l\times \CC P^l$ are $$
H^\ast(\CC P^l)=\Lambda(x)/(x^{l+1}),
$$ and $$H^\ast(\CC P^l\times \CC P^l)=\Lambda(u,v)/(u^{l+1}, v^{l+1}),$$ where $|x|=|u|= |v|=2$.
The kernel $\ker \mu_{n,Y}$ of the multiplication map $\mu_{n,Y}\colon H^\ast(Y)^{\otimes n}\to H^\ast(Y)$ 
contains the elements $\{u_i-u_j, v_i-v_j\mid 1\leq i, j\leq n\}$.
Under the image of $(f^\ast)^{\otimes n}$, these generators map as $$(f^\ast)^{\otimes n}(u_i-u_j)=(f^\ast)^{\otimes n}(v_i-v_j)=x_i-x_j.$$

We note that $(f^\ast)^{\otimes n}(\ker \mu_{n,Y})\subseteq \ker \mu_{n,X}$ which follows from the commutative square 
\[\begin{tikzcd}
	{H^\ast(Y)^{\otimes n}} & {H^\ast(Y)^{}} \\
	{H^\ast(X)^{\otimes n}} & {H^\ast(X)^{}}
	\arrow["{\mu_{n,Y}}", from=1-1, to=1-2]
	\arrow["{(f^\ast)^{\otimes n}}"', from=1-1, to=2-1]
	\arrow["{f^\ast}", from=1-2, to=2-2]
	\arrow["{\mu_{n,X}}"', from=2-1, to=2-2]
\end{tikzcd}\]

Since $\{x_i-x_{i-1}\mid 2\leq i\leq n\} $ is generating $\ker \mu_{n,X}$ 
and all of these generators belong to the image of $ (f^\ast)^{\otimes n}(\ker \mu_{n,Y})$, we obtain the equality $(f^\ast)^{\otimes n}(\ker \mu_{n,Y})= \ker \mu_{n,X}.$ 

We have already computed the nilpotency of the ideal generated by $\{x_i-x_{i-1}\mid 2\leq i\leq n\} $ in $\Lambda(x_1,\dots, x_n)/(x_i^{l+1}\mid 1\leq i \leq n)$, which is $nl.$
Therefore, by Proposition~\ref{prop: Tc for formal map}, we conclude that $$\TC_n(f_0)=nl.$$
\end{example}

In \cite{Murillo_wu}, Murillo and Wu introduced a notion of topological complexity for a map which is equivalent to Scott’s version \cite{Scott}.

\begin{definition}\label{defn: MW tc of map}
Given a continuous map $f\colon X\to Y$, the \emph{topological 
complexity } of  $f$, $\TC^{MW}(f)$, is the least non-negative integer $n\leq \infty$ such that $X\times X$  can be covered by $n+1$ open sets $\{U_i\}_{i=0}^n$  on each of which there is a continuous map $s_i\colon U_i\to X^I$ satisfying $$(f\times f)\pi s_i\simeq (f\times f)|_{U_i},$$ where $\pi\colon X^I\to X\times X$ is the path ﬁbration, $\pi(\gamma)=(\gamma(0),\gamma(1)).$ 
\end{definition}
As observed in \cite{Carrasquel2015}, the Murillo--Wu topological complexity is a particular instance of the relative sectional category. Explicitly, for a map $f \colon X \to Y$, $$\TC^{MW}(f)=\secat{f \times f}{(f \times f)\Delta_X}.$$

We now recover the computation of the topological complexity of a formal map given by Murillo and Wu \cite[Theorem~3.5]{Murillo_wu}, giving a direct computation without relying on the equivalence of different versions of the topological complexity of maps.

\begin{prop}
If $f\colon X\to Y$ is a formal map, then $$\TC^{MW}(f_0)=\nil\bigl(\ker \mu_X \cap \img(f^\ast \otimes f^\ast)\bigr),$$ where $\mu_X\colon H^\ast(X){\otimes }H^\ast(X)\to H^\ast(X)$ denotes the multiplication map.
\end{prop}
\begin{proof}
By the algebraic characterization of the relative sectional category for formal maps, we have $$\TC^{MW}(f_0)=\nil\bigl((f^\ast \otimes f^\ast)(\ker(\mu_X \circ (f^\ast \otimes f^\ast)))\bigr).$$ We claim that $$(f^\ast \otimes f^\ast)\bigl(\ker(\mu_X \circ (f^\ast \otimes f^\ast))\bigr)
= \ker \mu_X \cap \img(f^\ast \otimes f^\ast).$$
Let $f^\ast(a)\otimes f^\ast(b)\in \ker \mu_X \cap \img(f^\ast \otimes f^\ast)$ for some $a,b\in H^\ast(Y)$. Then $$\mu_X\bigl(f^\ast(a)\otimes f^\ast(b)\bigr)=f^\ast(ab)=0,$$ which implies $a\otimes b \in \ker(\mu_X \circ (f^\ast \otimes f^\ast))$. Hence $$f^\ast(a)\otimes f^\ast(b)\in (f^\ast \otimes f^\ast)\bigl(\ker(\mu_X \circ (f^\ast \otimes f^\ast))\bigr).$$

Conversely, if $\xi \in (f^\ast \otimes f^\ast)(\ker(\mu_X \circ (f^\ast \otimes f^\ast)))$, then $\xi=f^\ast(a)\otimes f^\ast(b)$ for some $a\otimes b \in \ker \mu_X \circ(f^\ast \otimes f^\ast).$  Then $f^\ast(a)\cdot f^\ast(b)=0,$ which means, $$\xi=f^\ast(a)\otimes f^\ast(b)\in \ker \mu_X \cap \img(f^\ast \otimes f^\ast).$$ The claim follows, completing the proof.
\end{proof}

We now give a class of maps for which the topological complexity can be found explicitly, generalizing \cite[Theorem 1.4]{JessupMurilloParent2012}.
\begin{theorem}\label{thm: codomain with odd homotopy}
    Let $f\colon X\to Y$ be a homotopy retraction and $\pi_\ast(Y)\otimes \QQ$ be finite-dimensional and concentrated in odd degrees. Then $\TC_n(f_0)=(n-1)k,$ where $k=\dim(\pi_\ast(Y)\otimes \QQ).$  \end{theorem}
\begin{proof}
    Let $\varphi\colon \Lambda V_Y\to \Lambda V_X$ be the minimal model of $f$. By Proposition~\ref{prop: rational version of tc}, we have $TC_n(f_0)=\secat{{\varphi}^{\otimes n}}{\mu_n}.$
    We may assume $\Lambda V_Y=\Lambda(x_1,\dots , x_k)$, where $x_i$ are generators of odd degrees. Let $x_{i,j}$ denote the element $1\otimes\cdots \otimes x_i\otimes\cdots \otimes 1$ where $x_i$ appears in $j$-th tensor factor. Then $$(\Lambda V_Y)^{\otimes n}=\Lambda(x_{i,j} \mid 1 \leq i \leq k, 1 \leq j \leq n).$$
    The kernel $\ker \mu_n$ of the multiplication map is the ideal generated by
    $$\{x_{i,j}-x_{i,1} \mid 1 \leq i \leq k, 2 \leq j \leq n\}. $$
We also note that $\nil(\ker \mu_n)=(n-1)k$ from a similar calculation as in the odd-dimensional case of Example~\ref{ex: Tc of sphere map}.
    
    By \cite[Theorem 15.11]{FelixHalperinThomas2001}, we have a natural isomorphism between the linear part of the minimal model and the dual of rational homotopy groups
    \[\begin{tikzcd}
	{V_Y} & {\hom(\pi_\ast (Y),\QQ)} \\
	{V_X} & {\hom(\pi_\ast (X),\QQ)}
	\arrow["\cong", from=1-1, to=1-2]
	\arrow["{Q(f)}"', from=1-1, to=2-1]
	\arrow["{\hom(\pi_\ast (f),\QQ)}", from=1-2, to=2-2]
	\arrow["\cong", from=2-1, to=2-2]
\end{tikzcd}\]
    By hypothesis, there is a map $i\colon Y\to X$ such that $fi\simeq \id_Y $. Thus, on rational homotopy groups, $f_\ast i_\ast=\id.$
    Thus, we have $\pi_\ast(f)\otimes \QQ$ is surjective, meaning its dual, the linear part $Q(f)$, is injective. Consequently, $\varphi$ itself is injective. By Lemma~\ref{lem: nilpotency preserved by mono}, we have $\nil({\varphi}^{\otimes n}(\ker \mu_n))=k(n-1)$.
    Applying Proposition~\ref{prop: compare Hsc msc sc nil} and Proposition~\ref{prop: compare secat vs sc} yields $\secat{{\varphi}^{\otimes n}}{\mu_n} \leq \nil({\varphi}^{\otimes n}(\ker \mu_n))$. Therefore, $$\secat{{\varphi}^{\otimes n}}{\mu_n} \leq k(n-1).$$

 To establish the lower bound, we use Proposition~\ref{prop: compare nilker vs secat}, which gives $$\nil ({{\varphi}^\ast}^{\otimes n-1} ((\Lambda V_Y)^{\otimes n-1}) \leq \cat (f_0^{n-1}).$$
   On cohomology, we have $i^\ast f^\ast=\id$, so $f^\ast=\varphi^\ast$ is injective. Therefore $\nil({{\varphi}^\ast}^{\otimes n-1} ((\Lambda V_Y)^{\otimes n-1})^+)=(n-1)k.$
   
    We now show that $\cat (f^{n-1})\leq \TC_n(f)$ similar to the result given in \cite[Proposition 3.8]{Scott}. 
   $j \colon U \hookrightarrow X^n$ is the inclusion of an open set that admits a map $s \colon U \to Y^I$ such that the following diagram commutes:
    \[\begin{tikzcd}
	U & {Y^I} \\
	{X^n} & {Y^n}
	\arrow["s", from=1-1, to=1-2]
	\arrow["j"', from=1-1, to=2-1]
	\arrow["{\pi_n}", from=1-2, to=2-2]
	\arrow["{f^n}"', from=2-1, to=2-2]
\end{tikzcd}\]
where the fibration $\pi_n \colon Y^I \to Y^n$ evaluates a path at $n$ equally spaced points. 
 Let $x_0\in X.$ Restricting the diagram to the slice $X^{n-1} \times \{x_0\}$ yields
\[\begin{tikzcd}
	{U|_{X^{n-1}\times \{x_0\}}} & {P_1Y} \\
	{{X^{n-1}\times \{x_0\}}} & {{Y^{n-1}\times \{f(x_0)\}}}
	\arrow["s", from=1-1, to=1-2]
	\arrow["j"', from=1-1, to=2-1]
	\arrow["{\pi_n}", from=1-2, to=2-2]
	\arrow["{f^n}"', from=2-1, to=2-2]
\end{tikzcd}\]
where $P_1Y$ denotes the space of all paths in $Y$ that end at $f(x_0).$
Since $P_1Y$ is contractible, the commutativity of the diagram implies that the restriction map $f^{n-1}\colon {U|_{X^{n-1}}}\to Y^{n-1}$ is nullhomotopic. Thus, we have $\cat (f^{n-1})\leq \TC_n(f)$.

Therefore, combining these bounds gives $$(n-1)k= \nil({{\varphi}^\ast}^{\otimes n-1} ((\Lambda V_Y)^{\otimes n-1})^+)\leq \cat (f_0^{n-1})\leq \TC_n(f_0).$$
     This implies $\TC_n(f_0)=(n-1)k$, completing the proof.
\end{proof}
\begin{example}
Let $f\colon  X\to K(G,m)$ be a retraction where $G=\ZZ^l\oplus (\text{torsion})$ and $m>1$ is odd. 
Clearly, $\pi_\ast(K(G,m))\otimes\QQ=\QQ^l$ which has dimension $l.$
By Theorem~\ref{thm: codomain with odd homotopy}, we have $\TC_n(f_0)=(n-1)l.$
\end{example}

\begin{remark}
A useful example of the relative sectional category is homotopic distance, which was introduced by Mac\'{\i}as-Virg\'os and   Mosquera-Lois \cite{MACÍAS–VIRGÓS_MOSQUERA–LOIS_2022} and measures how far two maps are from being homotopic.
As observed in \cite{Calc24}, the homotopic distance between two maps $f,g\colon X\to Y$ is given by $D(f,g)=\secat{(f\times g)\Delta_X}{\Delta_Y}$, where $\Delta_Z\colon Z\to Z\times Z$ is the diagonal map for $Z=X,Y$. For a pair of formal maps between rational spaces, we get an explicit description of the homotopic distance using similar algebraic machinery.
Specifically, if $f,g \colon X \to Y$ are formal maps between rational spaces, then $$D(f_0,g_0)=\nil ((\mu_X (f^\ast \otimes g^\ast))(\ker \mu_Y))$$ where $\mu_Z\colon H^\ast(Z)\otimes H^\ast(Z) \to H^\ast(Z)$ is the multiplication map for $Z=X,Y$. 
\end{remark}

\vspace{0.3em}
\noindent \textbf{Acknowledgment:}
The authors express their sincere gratitude to their research advisor, Prof. Rekha Santhanam, for her constant guidance, encouragement, and insightful discussions throughout this work; her perspective was instrumental in shaping the direction of this article. They also thank Professor Garc\'ia Calcines and the anonymous referee for their careful reading of the earlier version of the paper and for their valuable comments and suggestions, which have helped us improve the paper substantially.
Bittu Singh thanks Prof. Navnath Daundkar and Dr. Soumyadip Thandar for numerous valuable discussions. He gratefully acknowledges the support of the Prime Minister's Research Fellowship (PMRF ID 1302644), Government of India. 
Lekha Das gratefully acknowledges the support provided by IIT Bombay through its research grant.

\end{document}